\documentclass[11pt]{amsart}
\usepackage{amssymb, graphicx, amsthm}
\usepackage{epsfig}
\usepackage{verbatim}
\usepackage{amsfonts}
\usepackage{mathrsfs}
\usepackage{amsbsy}
\usepackage{graphicx}
\newcommand{\ds}{\displaystyle}
\newcommand{\wt}{\widetilde}
\newcommand{\ra}{{\rightarrow}}
\newcommand{\eproof}{\hfill\rule{2.2mm}{3.0mm}}
\newcommand{\esubproof}{\hfill$\square$}
\newcommand{\Proof}{\noindent {\bf Proof.~~}}
\newcommand{\D}{{\mathcal D}}
\newcommand{\TT}{{\mathcal T}}
\newcommand{\PP}{{\mathcal P}}
\newcommand{\E}{{\mathcal E}}
\newcommand{\F}{{\mathcal F}}
\newcommand{\G}{{\mathcal G}}
\newcommand{\M}{{\mathcal M}}
\newcommand{\B}{{\mathcal B}}
\newcommand{\RR}{{\mathcal R}}
\newcommand{\SSS}{{\mathcal S}}
\newcommand{\HH}{{\mathcal H}}
\newcommand{\CC}{{\mathcal C}}
\newcommand{\QQ}{{\mathcal Q}}
\newcommand{\Sq}{\Sigma_q}
\newcommand{\ST}{\{0,1\}}
\newcommand{\I}{{\mathcal I}}
\newcommand{\J}{{\mathcal J}}
\newcommand{\K}{{\mathcal K}}
\newcommand{\LL}{{\mathcal L}}
\newcommand{\W}{{\mathcal W}}
\newcommand{\V}{{\mathcal V}}
\newcommand{\U}{{\mathcal U}}
\newcommand{\Haus}{{\mathcal H}}
\newcommand{\Prob}{{\rm Prob}}
\newcommand{\z}{{\mathcal z}}
\newcommand{\m}{{\mathcal m}}
\newcommand{\R}{{\mathbb R}}
\newcommand{\Z}{{\mathbb Z}}
\newcommand{\T}{{\mathbb T}}
\newcommand{\C}{{\mathbb C}}
\newcommand{\Q}{{\mathbb Q}}
\newcommand{\N}{{\mathbb N}}
\renewcommand{\v}{{\bf v}}
\newcommand{\e}{{\bf e}}
\newcommand{\uu}{{\bf u}}
\newcommand{\ww}{{\bf w}}
\newcommand{\edge}{{\mapsto}_R}
\renewcommand{\i}{{\bf i}}
\renewcommand{\j}{{\bf j}}
\renewcommand{\k}{{\bf k}}
\renewcommand{\l}{{\bf l}}
\newcommand{\p}{{\bf p}}
\newcommand{\ep}{\varepsilon}
\newcommand{\wmod}[1]{\mbox{~(mod~$#1$)}}
\renewcommand{\eqref}[1]{(\ref{#1})}
\newcommand{\shsp}{\hspace{1em}}
\newcommand{\mhsp}{\hspace{2em}}
\newcommand{\grad}{\nabla}
\newcommand{\bfL}{{\boldsymbol{\lambda}}}
\newcommand{\bfO}{{\boldsymbol{\Omega}}}
\newtheorem{prop}{Proposition}[section]
\newtheorem{lem}[prop]{Lemma}

\newtheorem{thm}{Theorem}[section]

\usepackage{amscd}
\renewcommand{\theequation}{\thesection.\arabic{equation}}
\renewcommand{\thefigure}{\arabic{figure}}

\begin{document}

\title[Asymptotic translation lengths of point-pushing pseudo-Anosov maps]{On the minimum of asymptotic translation lengths of point-pushing pseudo-Anosov maps on one punctured Riemann surfaces}
\author[C. Zhang]{C. Zhang}
\date{January 20, 2013}
\thanks{ }

\address{Department of Mathematics \\ Morehouse College
\\ Atlanta, GA 30314, USA.}
\email{czhang@morehouse.edu}


\maketitle 

\begin{abstract}
We show that the minimum of asymptotic translation lengths of all point-pushing pseudo-Anosov maps on any one punctured Riemann surface is one.
\end{abstract}

\section{Introduction and main results}
\setcounter{equation}{0}

Let $S$ be a closed Riemann surface of genus $p$ with $n$ points removed. Assume that $3p-4+n>0$. One can associate to $S$ a curve complex $\mathcal{C}(S)$ which is equipped with a path metric $d_{\mathcal{C}}$. Let $\mathcal{C}_0(S)$ denote the set of vertices of $\mathcal{C}(S)$ that can be identified with the set of simple closed geodesics on $S$.  See Section 2 for the definitions and terminology. 

Following Farb--Leininger--Margalit \cite{F-L-M}, for any $u\in \mathcal{C}_0(S)$, and any pseudo-Anosov map $f$ of $S$, we can define $\tau_{\mathcal{C}}(f)$ as 
\begin{equation}\label{ASYM}
\tau_{\mathcal{C}}(f)=\liminf_{m\rightarrow \infty}\frac{\ d_{\mathcal{C}}\left(u,f^m(u)\right)\ }{m}. 
\end{equation}
It is known that $\tau_{\mathcal{C}}(f)$ does not depend on choices of vertices $u$ in $\mathcal{C}_0(S)$ and is called the asymptotic translation length for the action of $f$ on $\mathcal{C}(S)$. Bowditch \cite{Bo} proved that $\tau_{\mathcal{C}}(f)$ for all pseudo-Anosov maps are rational numbers. 

Let Mod$(S)$ denote the mapping class group of $S$, and let $H\subset \mbox{Mod}(S)$ be a subgroup. Denote by 
$$
L_{\mathcal{C}}(H)=\mbox{inf }\{\tau_{\mathcal{C}}(f): \ \mbox{for all pseudo-Anosov elements in } H\}.
$$
By Masur--Minsky \cite{M-M}, there is a positive lower bound for $L_{\mathcal{C}}(H)$ that depends only on $(p,n)$. For a closed surface $S$ of genus $p>1$, Theorem 1.5 of \cite{F-L-M} asserts that 
$$
L_{\mathcal{C}}\left(\mbox{Mod}(S)\right)<\frac{\ 4 \log \left( 2+\sqrt{3}   \right)\ }{p \log \left( p-\frac{1}{2}   \right)}. 
$$
This particularly implies that $L_{\mathcal{C}}\left( \mbox{Mod}(S)\right)\rightarrow 0$ as $p\rightarrow +\infty$.  The lower and upper bounds for $L_{\mathcal{C}}\left(\mbox{Mod}(S)\right)$ were improved as
$$
\frac{1}{162\left( 2p-2  \right)^2+30 \left(2p-2   \right)}< L_{\mathcal{C}}\left( \mbox{Mod}(S)\right)\leq \frac{4}{p^2+p-4}
$$
by a result of Gadre--Tsai \cite{G-T}.  

The estimations of $L_{\mathcal{C}}(H)$ for certain subgroups $H$ of $\mbox{Mod}(S)$ were also considered in \cite{F-L-M}. Let $\Gamma_0$ be the fundamental group of $S$. For any $k\geq 1$, let $\Gamma_k$ be the $k$th term of the lower central series for $\Gamma_0$. This chain of subgroups forms a filtration. Denote by $\mathscr{N}_k$ the kernel of the natural homomorphism of $\mbox{Mod}(S)$ onto Out$(\Gamma/\Gamma_k)$. Then for the sequence of the subgroups $\mathscr{N}_k$, Theorem 6.1 of \cite{F-L-M} states that for any $k$, a similar phenomenon emerges. That is, 
$$
L_{\mathcal{C}}(\mathscr{N}_k(S))\rightarrow 0\ \ \mbox{as}\ \ p\rightarrow +\infty. 
$$

In this paper, we are mainly concerned with the case in which $S$ contains only one puncture $x$. Then the subgroup $\mathscr{F} \subset \mbox{Mod}(S)$ that consists of mapping classes projecting to the trivial mapping class on $\tilde{S}:=S\cup \{x\}$ is highly non trivial and is isomorphic to the fundamental group $\pi_1(\tilde{S},x)$. It is well-known (Kra \cite{Kr}) that $\mathscr{F}$ contains infinitely many pseudo-Anosov elements, and the conjugacy class of a primitive pseudo-Anosov element of $\mathscr{F}$ can be determined by an oriented primitive filling closed geodesic $\tilde{c}$ on $\tilde{S}$ in the sense that $\tilde{c}$ intersects every simple closed geodesic on $\tilde{S}$. 

In contrast to the above estimations for $L_{\mathcal{C}}(H)$ for various subgroups $H$ of $\mbox{Mod}(S)$, in the case where $H=\mathscr{F}$, we can view $L_{\mathcal{C}}(\mathscr{F})$ as a function of $(p,n)$, and see that $L_{\mathcal{C}}(\mathscr{F})$ performs quite differently than 
$L_{\mathcal{C}}( \mbox{Mod}(S))$ and $L_{\mathcal{C}}(\mathscr{N}_k(S))$. The main purpose of this paper is to prove the following result.
\begin{thm} \label{T1}
For any type $(p,1)$ with $p>1$, $L_{\mathcal{C}}(\mathscr{F})=1$. 
\end{thm}

We may find a filling closed geodesic $\tilde{c}$ on $\tilde{S}$ and a vertex $\tilde{u}\in \mathcal{C}_0(\tilde{S})$ so that $\tilde{u}$ intersects $\tilde{c}$ only once. Let 
$u\in \mathcal{C}_0(S)$ be the vertex obtained from $\tilde{u}$ by removing $x$. Let $f\in \mathscr{F}$ be a pseudo-Anosov element obtained from pushing $x$ along $\tilde{c}$ (see Theorem 2 of \cite{Kr}). From \cite{CZ5},  we know that  $\{u,f(u)\}$ forms the  boundary of an $x$-punctured cylinder on $S$. This means that $i(u,f(u))=0$, where and below $i(\alpha, \beta)$ denotes the geometric intersection number between two vertices $\alpha,\beta\in \mathcal{C}_0(S)$. Note, since $f$ is a homeomorphism of $S$, that $i(f(u),f^2(u))=0$. Hence $f(u)$ is disjoint from both $u$ and $f^2(u)$. By the definition of $d_{\mathcal{C}}$, we have $d_{\mathcal{C}}\left(u, f^2(u)\right)\leq 2$. Hence from the construction of $f$, $u$ intersects $f^2(u)$, which implies that $d_{\mathcal{C}}\left(u, f^2(u)\right)>1$. We conclude that $d_{\mathcal{C}}\left(u, f^2(u)\right)=2$. Now we modify the argument of \cite{F-L-M}. Since $f$ is a homeomorphism of $S$, we obtain 
$$
d_{\mathcal{C}}\left(f^{2m}(u), f^{2m-2}(u)\right)=2\ \  \mbox{for}\ \  m=1,2,\cdots.
$$
Now the triangle inequality yields 
$d_{\mathcal{C}}\left(f^{2m}(u), u\right)\leq 2m$, which says 
$$
\frac{\ d_{\mathcal{C}}\left(f^{2m}(u), u\right)\ }{2m}\leq 1
$$ 
for all positive integers $m$. It follows from (\ref{ASYM}) that $\tau_{\mathcal{C}}(f)\leq 1$ and thus that  $L_{\mathcal{C}}(\mathscr{F})\leq 1$. The assertion that $L_{\mathcal{C}}(\mathscr{F})\geq 1$ follows from the following result.

\begin{thm}\label{T2}
Let $S$ be of type $(p,1)$ with $p>1$ and let $f\in \mathscr{F}$ be a pseudo-Anosov element. Then there is $u\in \mathcal{C}_0(S)$ such that for any integer $m$ with $|m|\geq 3$, we have 
\begin{equation}\label{EQ}
d_{\mathcal{C}}\left( u, f^m(u)  \right)\geq |m|. 
\end{equation}
\end{thm}

\noindent {\em Remark. } Theorem \ref{T2} is compared with Proposition 3.6 of \cite{M-M}, which states that there is a constant $c=c(p,n)$, $c>0$, such that $d_{\mathcal{C}}(u, f^m(u))\geq c|m|$ for all pseudo-Anosov maps $f$ and all $u\in \mathcal{C}_0(S)$. The quantitative estimation for $c$ is, however, largely unknown. 

\noindent {\em Outline of Proof. }  Let $\mathbf{H}$ be a hyperbolic plane and $\varrho:\mathbf{H}\rightarrow \tilde{S}$ the universal covering map with a covering group $G$. Then $G$ is purely hyperbolic. There is an essential hyperbolic element $g\in G$ that corresponds to $f$ (Theorem 2 of \cite{Kr}). 


In the case where $S$ contains only one puncture $x$, all vertices $u$ in $\mathcal{C}_0(S)$ are non-preperipheral, in the sense that $u$ is homotopic to a non-trivial simple closed geodesic on $\tilde{S}$ as $x$ is filled in. Thus,
for each  vertex $u_0\in \mathcal{C}_0(S)$, there defines a configuration $(\tau_0, \Omega_0, \mathscr{U}_0)$ that corresponds to $u_0$. See Section 2 for expositions. 

Note that $\tau_{\mathcal{C}}(f)$ does not depend on choices of $u\in \mathcal{C}_0(S)$. A non-preperipheral vertex $u_0\in \mathcal{C}_0(S)$ can be selected so that $\Omega_0\cap \mbox{axis}(g)\neq \emptyset$ and $i(\varrho(\mbox{axis}(g)), \tilde{u})\geq 2$, where we use the similar notation $i(\tilde{c},\tilde{u})$ to denote the intersection number between a vertex $\tilde{u}$ and a filling curve $\tilde{c}$ (we always assume that $\tilde{u}$ intersects $\tilde{c}$ at non self-intersection points of $\tilde{c}$ by performing a small perturbation if necessary). For $m\geq 3$, let $u_m$ be the geodesic homotopic to the image curve $f^m(u_0)$. Suppose that 
\begin{equation}\label{PATH}
[u_0, u_1, \cdots, u_s, u_m]
\end{equation}
is an arbitrary geodesic path in the 1-skeleton of $\mathcal{C}(S)$ that connects $u_0$ and $u_m$ with a minimum number of sides. Then all $u_j$, $1\leq j\leq s$, are non-preperipheral, which allows us to obtain the configurations $(\tau_j, \Omega_j, \mathscr{U}_j)$ determined by the vertices $u_j$. Note that the sequence $\mathbf{H}\backslash \Delta_j'$ (See Fig. 1 and (\ref{COM}) for the construction of $\Delta_j'$) monotonically moves down towards the attracting fixed point $A$ of $g$, and the optimal scenario is so does the sequence $\Omega_j$. In case this occurs, we will show that the average rate of the movement of $\Omega_j$ towards $A$ is no faster than that of $\mathbf{H}\backslash \Delta_j'$. This leads to that $\Omega_j\cap \Delta_m'\neq \emptyset$ for $j\leq m-2$, which will imply that $u_j$ intersects $u_m$ as long as $j\leq m-2$. It follows that $s\geq m-1$ and thus that $d_{\mathcal{C}}\left( u_0, u_m  \right)\geq m$. If $m$ is negative and $m\leq -3$, the proof is similar. 

\section{Curve complex and tessellations in hyperbolic plane}
\setcounter{equation}{0}

Let $S$ be of type $(p,n)$. Due to Harvey \cite{H}, one can define the curve complex $\mathcal{C}(S)$ of dimension $3p-4+n$ as the following simplicial complex: vertices of $\mathcal{C}(S)$ are simple closed geodesics, and $k$-dimensional simplicies of $\mathcal{C}(S)$ are collections of $(k+1)$-tuples $\{u_0, u_1,\ldots, u_k\}$ of disjoint simple closed geodesics on $S$. Let $\mathcal{C}_k(S)$ denote the $k$-skeleton of $\mathcal{C}(S)$. We then introduce a metric $d_{\mathcal{C}}$, called the path metric, in the following way. First we make each simplex Euclidean with side length one, then for any vertices $u,v\in \mathcal{C}_0(S)$, we declare the distance $d_{\mathcal{C}}(u,v)$ between $u$ and $v$ to be the smallest number of edges connecting $u$ and $v$.  The curve complex $\mathcal{C}(\tilde{S})$ is similarly defined. 

Throughout the rest of the paper we assume that $S$ is a closed Riemann surface minus one point $x$. 
By forgetting the puncture $x$, we can define a fibration tructure $\mathcal{C}(S)\rightarrow \mathcal{C}(\tilde{S})$ that admits a global section (since any vertex in  $\mathcal{C}_0(\tilde{S})$ can be naturally thought of as a vertex in $\mathcal{C}_0(S)$). For each $\tilde{\varepsilon}\in \mathcal{C}_0(\tilde{S})$, let $F_{\tilde{\varepsilon}}$ be the fiber over $\tilde{\varepsilon}$ that consists of $u\in \mathcal{C}_0(S)$ for which $\tilde{u}=\tilde{\varepsilon}$, where $\tilde{u}$ is homotopic to $u$ if $u$ is viewed as curves on $\tilde{S}$.  

Fix $\tilde{\varepsilon}\in \mathcal{C}_0(\tilde{S})$. Let $\varrho^{-1}(\tilde{\varepsilon})$ denote the collection of geodesics $\hat{\varepsilon}$ in $\mathbf{H}$ such that $\varrho(\hat{\varepsilon})=\tilde{\varepsilon}$. Since $\tilde{\varepsilon}$ is simple, all geodesics in $\varrho^{-1}(\tilde{\varepsilon})$ are mutually disjoint. It is also clear that   $\varrho^{-1}(\tilde{\varepsilon})$ gives rise to a partition of $\mathbf{H}$. Let $\mathscr{R}_{\tilde{\varepsilon}}$ be the set of components of $\mathbf{H}\backslash \varrho^{-1}(\tilde{\varepsilon})$. By Lemma 2.1 of \cite{CZ6}, there is a bijection $\chi: \mathscr{R}_{\tilde{\varepsilon}}\rightarrow F_{\tilde{\varepsilon}}$. Each $\Omega\in \mathscr{R}_{\tilde{\varepsilon}}$ tessellates the hyperbolic plane $\mathbf{H}$ under the action of $G$. See \cite{CZ6} for more information on the tessellation. 
 
Let $\Omega\in \mathscr{R}_{\tilde{\varepsilon}}$. The Dehn twist $t_{\tilde{\varepsilon}}$ can be lifted to a map $\tau:\mathbf{H}\rightarrow \mathbf{H}$ so that the restriction $\tau|_{\Omega}=\mbox{id}$. Observe that the complement of the closure of $\Omega$ is a disjoint union of half-planes. Each such half plane $\Delta$ includes infinitely many geodesics in $\varrho^{-1}(\tilde{\varepsilon})$, and no geodesics in $\varrho^{-1}(\tilde{\varepsilon})$ are contained in $\Omega$. 
Thus, there defines infinitely many half planes contained in $\Delta$. Let $\mathscr{U}$ be the collection of all such half planes. Obviously $\mathscr{U}$ is a partially ordered set defined by inclusion. Maximal elements of $\mathscr{U}$ are called first order elements ($\Delta$ is one of them), elements of $\mathscr{U}$ that are included in a maximal element but are not included in any other elements of $\mathscr{U}$ are called second order elements, and so on. We see that for any element $\Delta_n$ of order $n$ with $n\geq 2$, there is a unique element $\Delta_{n-1}$ of order $n-1$ such that $\Delta_n\subset \Delta_{n-1}$. Conversely, for each $\Delta_{n-1}\in \mathscr{U}$ of order $n-1$, there are infinitely many disjoint elements $\Delta_n\in \mathscr{U}$ of order $n$ that are contained in $\Delta_{n-1}$. 

Each maximal element $\Delta$ is an invariant half plane under the action of $\tau$; and element $\Delta'\subset \Delta$ of any other order is not $\tau$-invariant, but $\tau$ sends $\Delta'$ to an element $\Delta''\subset \Delta$ of the same order. The map $\tau$ is quasiconformal and extends to a quasisymmetric map on $\mathbf{S}^1$. See \cite{CZ4} for more details. 

Let $\Omega\in \mathscr{R}_{\tilde{\varepsilon}}$ be such that $\chi(\Omega)=u$ for some $u\in \mathcal{C}_0(S)$. We call the triple $(\tau,\Omega,\mathscr{U})$ the configuration corresponding to $u$. Write $\tau_u=\tau$, $\Omega_u=\Omega$ and $\mathscr{U}_n=\mathscr{U}$ to emphasize this correspondence. 

For $i=1,2$, let $u_i\in \mathcal{C}_0(S)$, and let $(\tau_i, \Omega_i,\mathscr{U}_i)$ be the configurations corresponding to $u_i$. If $\tilde{u}_1=\tilde{u}_2=\tilde{\varepsilon}$, i.e., $u_i\in F_{\tilde{\varepsilon}}$, then $\Omega_i\in \mathscr{R}_{\tilde{\varepsilon}}$. Since $\mathscr{R}_{\tilde{\varepsilon}}$ is $G$-invariant, there is $h\in G$ such that $h(\Omega_1)=\Omega_2$. Obviously, $\Omega_1=\Omega_2$ if and only if $h=\mbox{id}$. Suppose now that $\Omega_1\neq \Omega_2$. 
Then $\Omega_1$ is disjoint from $\Omega_2$, and there is a path $\Gamma$ in $F_{\tilde{\varepsilon}}$ connecting $u_1=\chi(\Omega_1)$ and $u_2=\chi(\Omega_2)$ (Proposition 2.4 of \cite{CZ6}). Unfortunately, there is no guarantee that $\Gamma$ is a geodesic path in $\mathcal{C}_1(S)$. When $\Omega_1$ and $\Omega_2$ are adjacent, i.e., $\bar{\Omega}_1\cap \bar{\Omega}_2$ is a geodesic  in $\varrho^{-1}(\tilde{\varepsilon})$, then it can be shown that $\{\chi(\Omega_1), \chi(\Omega_2)\}$ forms the boundary of an $x$-punctured cylinder on $S$. In particular, we assert that $d_{\mathcal{C}}(\chi(\Omega_1), \chi(\Omega_2))=1$. See \cite{CZ6} for more details.  

In the case where $\tilde{u}_1\neq \tilde{u}_2$, the relationship between $\mathscr{R}_{\tilde{u}_1}$ and $\mathscr{R}_{\tilde{u}_2}$ is more complicated. However, if there are $u_1\in F_{\tilde{u}_1}$ and $u_2\in F_{\tilde{u}_2}$ such that $u_1$ is disjoint from $u_2$, then $\tilde{u}_1$ is disjoint from $\tilde{u}_2$, which implies that $\varrho^{-1}(\tilde{u}_1)$ is disjoint from $\varrho^{-1}(\tilde{u}_2)$. We have the following result which was proved in \cite{CZ5}. 
\begin{lem}\label{L2.1}
Suppose that $u_1, u_2$ are disjoint with $\tilde{u}_1\neq \tilde{u}_2$. Then $\Omega_1\cap \Omega_2\neq \emptyset$. Moreover, each maximal element of $\mathscr{U}_1$ contains or is contained in a maximal element of $\mathscr{U}_2$, and vise versa.  
\end{lem}
\noindent {\em Remark. } If a maximal element $\Delta_1\in \mathscr{U}_1$ contains a maximal element of $\mathscr{U}_2$, then $\Delta_1$ contains infinitely many maximal elements of $\mathscr{U}_2$; but if $\Delta_1\in \mathscr{U}_1$ is contained in a maximal element $\Delta_2$ of $\mathscr{U}_2$, then such a $\Delta_2$ is unique. The same is true for maximal elements of $\mathscr{U}_2$.  

\medskip

By assumption, $S$ contains only one puncture, which means that any mapping class must fix the puncture. It turns out that the $x$-pointed mapping class group (which is defined as a group that consists of mapping classes fixing $x$) is the same as the ordinary mapping class group Mod$(S)$. It is well-known (Theorem 4.1 and Theorem 4.2 of Birman \cite{Bir}) that there exists an exact sequence 
\begin{equation}\label{EXACT}
0 \longrightarrow \pi_1(\tilde{S},x)\longrightarrow \mbox{Mod}(S) \longrightarrow \mbox{Mod}(\tilde{S}) \longrightarrow 0,
\end{equation}
which defines an injective map $\psi: G\rightarrow \mbox{Mod}(S)$ (since $G$ is canonically isomorphic to $\pi_1(\tilde{S},x)$). Let $Q(G)$ be the group of quasiconformal automorphisms of $\mathbf{H}$. We introduce an equivalence relation $``\sim"$ in $Q(G)$ as follows. Two element $w_1,w_2\in  Q(G)$ are declared to be equivalent (write as $w_1\sim w_2$) if $w_1=w_2$ on $\partial \mathbf{H}=\mathbf{S}^1$. The quotient group $Q(G)/\! \sim$ is isomorphic to 
$\mbox{Mod}(S)$ via a ``Bers isomorphism" $\varphi$ \cite{Bers1}. Notice that $G$ is naturally regarded as a normal subgroup of $Q(G)/\! \sim$, $\varphi$ restricts to the injective map $\psi$ defined by (\ref{EXACT}), and we have $\varphi(G)=\psi(G)=\mathscr{F}$. For each element $h\in G$, let $h^*\in \mathscr{F}\subset \mbox{Mod}(S)$ denote the mapping class $\varphi(h)=\psi(h)$.  

\section{Partitions and regions in hyperbolic plane determined by vertices}
\setcounter{equation}{0}

Let $f\in \mathscr{F}$ be a pseudo-Anosov element. By Theorem 2 of \cite{Kr}, there is $g\in G$ such that $g^*=f$ and $g$ is an essential hyperbolic element, which means that the projection $\tilde{c}:=\varrho(\mbox{axis}(g))$ is an oriented filling closed geodesic on $\tilde{S}$, where $\mbox{axis}(g)$ denotes the axis of $g$ which is an invariant geodesic in $\mathbf{H}$ under the action of $g$. 

Choose $\tilde{u}_0\in \mathcal{C}_0(\tilde{S})$ so that $i(\tilde{u}_0,\tilde{c})\geq 2$ (there are infinitely many such $\tilde{u}_0$). Let $\Omega_0\in \mathscr{R}_{\tilde{u}_0}$ be such that $\Omega_0\cap \mbox{axis}(g)\neq \emptyset$. Then $\Omega_0$ determines a configuration $(\tau_0, \Omega_0,\mathscr{U}_0)$ that corresponds to a vertex $\chi(\Omega_0)=u_0\in F_{\tilde{u}_0}\subset \mathcal{C}_0(S)$.

By Lemma 3.1 of \cite{CZ5}, $\mbox{axis}(g)$ can not be completely included in $\Omega_0$, which means that there are maximal elements $\Delta_0,\Delta_0^*\in \mathscr{U}_0$ such that $\mbox{axis}(g)$ crosses both $\Delta_0$ and $\Delta_0^*$.  We may assume that $\Delta_0$ and $\Delta_0^*$ cover attracting and repelling fixed points of $g$, respectively. $\Delta_0$ and $\Delta_0^*$ are shown in Fig. 1. 

For $m\geq 3$, let $u_m$ denote the geodesic homotopic to the image of $u_0$ under the map $f^m$. Then $u_m$ is also a non-preperipheral geodesic and 
$$
(\tau_m,\Omega_m,\mathscr{U}_m):=(g^m\tau_0 g^{-m}, g^m(\Omega_0), g^m(\mathscr{U}_0))
$$ 
is the configuration corresponding to $u_m$. In particular, $\Delta_m':=g^m(\Delta_0^*)$ is a maximal element of  $\mathscr{U}_m$ that covers the repelling fixed point $B$ of $g$. $\Delta_m'$ is also drawn in Fig. 1.

\bigskip
 
\unitlength 1mm 
\linethickness{0.4pt}
\ifx\plotpoint\undefined\newsavebox{\plotpoint}\fi 


 In what follows, we use the symbol $\overline{P_iQ_i}$ to denote the geodesic in $\mathbf{H}$ connecting points $P_i$ and $Q_i$ on $\mathbf{S}^1$. Also, for any two non-antipodal points $X,Y\in \mathbf{S}^1$, let $(XY)$ denote the unoriented smaller arc on $\mathbf{S}^1$ connecting $X$ and $Y$. Likewise, we use $(XZ1\cdots Z_nY)$ to denote the arc on $\mathbf{S}^1$ that connects $X$ and $Y$ and passes through points $Z_1,\cdots, Z_n$ in order on $\mathbf{S}^1$. 

We thus have $\overline{P_0Q_0}=\partial \Delta_0$.  Denote by
\begin{equation}\label{COM}
 \Delta_j'=g^j(\Delta_0^*) \ \ \mbox{for}\ j=1,2,\cdots, m,
\end{equation}
and let $\overline{P_jQ_j}=\partial \Delta_j'$. By inspecting Fig. 1,  we find that for $1\leq j\leq m-1$,
\begin{equation}\label{JHG}
g(\overline{P_jQ_j})=\overline{P_{j+1}Q_{j+1}}
\end{equation}
and that $\overline{P_jQ_j}$ is disjoint from $\overline{P_{j+1}Q_{j+1}}$. Furthermore, for $1\leq j\leq m-2$, 
\begin{equation}\label{JHG1}
g(P_iP_{i+1})=(P_{i+1}P_{i+2})\ \ \mbox{and}\ \  g(Q_iQ_{i+1})=(Q_{i+1}Q_{i+2}).
\end{equation}
It is also clear that $\overline{P_jQ_j}$ lies above $\overline{P_kQ_k}$ whenever $k>j\geq 1$. Since $i(\tilde{c},\tilde{u})\geq 2$, we assert that $\overline{P_0Q_0}$ lies above $\overline{P_1Q_1}$ ($\overline{P_0Q_0}=\overline{P_1Q_1}$ if and only if $i(\tilde{u},\tilde{c})=1$). See Fig. 1. All these geodesics $\overline{P_jQ_j}$ give rise to a partition of the hyperbolic plane $\mathbf{H}$. Note that $\Omega_0$ is the complement of all maximal elements of $\mathscr{U}_1$. We have $\Omega_0\subset \mathbf{H}\backslash (\Delta_0\cup \Delta_0^*)$ and $\Omega_m\subset \mathbf{H}\backslash \Delta_m'$.
 
 Suppose that a geodesic path  (\ref{PATH}) in $\mathcal{C}_1(S)$ connects $u_0$ and $u_m$, which tells us that $d_{\mathcal{C}}(u_j,u_{j+1})=1$ for $j=0,\cdots,s-1$, and $d_{\mathcal{C}}(u_s,u_{m})=1$. We need to show that $s\geq m-1$.
 
Since all $u_j$ are non-preperipheral, we can obtain the configurations $(\tau_j,\Omega_j,\mathscr{U}_j)$ corresponding to those $u_j$. Fix $k$ with $1\leq k\leq s$. A region $\Omega_j$, $1\leq j\leq s$, is called to be located at level $k$ if $\Delta_j=\Delta_k'$ for some maximal element $\Delta_j\in \mathscr{U}_j$. Similarly, $\Omega_j$ is called to be located above level $k$ if $\Omega_j\cap \Delta_k'\neq \emptyset$. Fig. 2 demonstrates the situation where $\Omega_j$ is located at level $k$, while Fig. 3, 4, 5 and 6 are all possible cases where $\Omega_j$ are located above level $k$. 

\bigskip

\unitlength 1mm 
\linethickness{0.4pt}
\ifx\plotpoint\undefined\newsavebox{\plotpoint}\fi 


By Lemma 3.1 of \cite{CZ5}, $\mbox{axis}(g)$ is not included in any $\Omega_j$. That is, either $\mbox{axis}(g)$ is contained in a maximal element of $\mathscr{U}_j$, or $\mbox{axis}(g)$ intersects $\partial \Delta_j$ and  $\partial \Delta_j^*$
for maximal elements $\Delta_j$ and $\Delta_j^*$ of $\mathscr{U}_j$. In both case, we may find a maximal $\Delta_j\in \mathscr{U}_j$, shown in Fig. 3, 4, 5, or 6,  that covers the attracting fixed point $A$ of $g$. 

\begin{lem}\label{L1}
Suppose that $\Omega_j$ is located above level $k$ with $k\leq m-1$ (Fig. $3, 4, 5, 6$). Let $\Delta_j\in \mathscr{U}_j$ be the maximal element that covers the attracting fixed point of $g$. Then at least one point of $\{X_j,Y_j\}:=\partial \Delta_j\cap \mathbf{S}^1$, $X_j$ say, lies above $\overline{P_{k+1}Q_{k+1}}$.
\end{lem}
\begin{proof}
By assumption, $\Omega_j\cap \Delta_k'\neq \emptyset$. If $\Omega_j\subset \Delta_k'$ (Fig. 6), then for the $\Delta_j$ shown in Fig. $6$, $\{X_j,Y_j\}$ both lie above $\overline{P_kQ_k}$. So both $\{X_j,Y_j\}$ lie above $\overline{P_{k+1}Q_{k+1}}$.  Suppose now that $\Omega_j$ is not a subset of $\Delta_k'$ and $\overline{P_kQ_k}$ crosses $\Delta_j$ (Fig. 3, 4), then we see that $X_j$ lies above $\overline{P_{k}Q_{k}}$. In particular, $X_j$ lies above $\overline{P_{k+1}Q_{k+1}}$. 

It remains to consider the case where $\overline{P_{k}Q_{k}}$ is disjoint from $\Delta_j$ (Fig. 5). Then $\Delta_j$ lies below $\overline{P_{k}Q_{k}}$ and intersects $\mbox{axis}(g)$. If both $\{X_j,Y_j\}$ lie below $\overline{P_{k+1}Q_{k+1}}$, then by Lemma 2.1 of \cite{CZ3}, (\ref{JHG}) and (\ref{JHG1}) we can find a maximal element $\Delta_j^*\in \mathscr{U}_j$ that covers $\Delta_k'$. Note that $\Omega_j\subset \mathbf{H}\backslash (\Delta_j\cup \Delta_j^*)$. We conclude that $\Omega_j$ is disjoint from  $\Delta_k'$. This is a contradiction. 
\end{proof}


\unitlength 1mm 
\linethickness{0.4pt}
\ifx\plotpoint\undefined\newsavebox{\plotpoint}\fi 


\begin{lem}\label{L2}
Suppose that $\Omega_j$ is located above level $k$ for $1\leq k\leq m-1$. Let $\Delta_j\in \mathscr{U}_j$ be the maximal element obtained from Lemma $\ref{L1}$. Then we have {\rm (i)} $\Delta_j$ is not contained in $\Delta_m'$,  {\rm (ii)} $\Delta_j$ is not disjoint from $\mbox{axis}(g)$, and  {\rm (iii)} $\Delta_j\cap \Delta_m'\neq \emptyset$.   
\end{lem}

\begin{proof}
Properties (i) and (ii) follow directly from the construction of $\Delta_j$ (by noting that $\Delta_j$ covers the attracting fixed point of $g$ while $\Delta_m'$ does not). For (iii),
we write $\{X_j,Y_j\}=\partial \Delta_j$.  By Lemma \ref{L1}, at least one point of $\{X_j,Y_j\}$, $X_j$ say, lies above $\overline{P_mQ_m}$. If both $X_j$ and $Y_j$ lie above  $\overline{P_{m}Q_m}$ (Fig. 6 with $k=m$), then $\Delta_j$ satisfies the conditions (i)-(iii) of the lemma. We are done. If only $X_j$ lies above  $\overline{P_mQ_m}$, then either $\Delta_j\supset \mbox{axis}(g)$ (Fig. 4 with $k=m$), in which case, $\Delta_j$ satisfies the conditions (i)-(iii) of the lemma), or $X_j$ and $Y_j$ are separated by $\mbox{axis}(g)$ (Fig. 3 with $k=m$), in which case,  $\partial \Delta_j\cap \mbox{axis}(g)\neq \emptyset$. It is easy to see that $\Delta_j$ is not contained in $\Delta_m'$ and $\Delta_j\cap \Delta_m'\neq \emptyset$. 

If both $X_j, Y_j$ lie below $\overline{P_kQ_k}$ (Fig. 5), by Lemma 3.1, at least one point of $\{X_j, Y_j\}$ lies above $\overline{P_{k+1}Q_{k+1}}$. Since $k+1\leq m$, we conclude that $\Delta_j\cap \Delta_m'\neq \emptyset$ and thus conditions (i)-(iii) remains valid.  
\end{proof}

\begin{lem}\label{L3}
If $\Omega_j$ is located above level $m-1$, or at level $m-2$, then $d_{\mathcal{C}}(u_j, u_m)\geq 2$. 
\end{lem}
\begin{proof}
First assume that $\Omega_j$ is located above level $m-1$. By Lemma \ref{L2}, there is a maximal $\Delta_j\in \mathscr{U}_j$ such that $\Delta_j$ is not contained in $\Delta_m'$ and $\Delta_j\cap \Delta_m'\neq \emptyset$. If $\partial \Delta_j\cap \partial \Delta_m'\neq \emptyset$, then $\tilde{u}_j$ intersects $\tilde{u}_m$, where $\tilde{u}_j$ is the geodesic on $\tilde{S}$ homotopic to $u_j$ if $u_j$ is viewed as a curve on $\tilde{S}$. Hence $u_j$ intersects $u_m$ and the assertion follows. 

Assume now that  $\partial \Delta_j\cap \partial \Delta_m'= \emptyset$. Then $\Delta_j\cap \mathbf{S}^1\supset (P_mAQ_m)$. In this case, $\Omega_j\subset \mathbf{H}\backslash (\Delta_j\cup \Delta_j^*)$ is disjoint from $\mathbf{H}\backslash \Delta_m'$. But $\Omega_m\subset \mathbf{H}\backslash \Delta_m'$. Hence $\Omega_j$ is disjoint from $\Omega_m$. It follows from Lemma \ref{L2.1} that $d_{\mathcal{C}}(u_j, u_m)\geq 2$.

Now suppose that $\Omega_j$ is located at level $m-2$ (Fig. 2 with $k=m-2$). Then there is maximal $\Delta_j\in \mathscr{U}_j$ such that $\Delta_j=\Delta_{m-2}'$. Again, by Lemma 2.1 of \cite{CZ3}, there is a maximal $\Delta_j^*\in \mathscr{U}_j$, shown in Fig. 2, so that $\Delta_j^*$ is  disjoint from $\Delta_j$, such that $\partial \Delta_j^*$ intersects $\mbox{axis}(g)$ and  $\Delta_j^*$ contains $\mathbf{H}\backslash \Delta_{m-1}'$. In particular, we see that $\Delta_j^*\cap \Delta_m'\neq \emptyset$. The assertion follows from Lemma \ref{L2.1}.
\end{proof}
\noindent {\em Remark. } The bound $m-2$ is optimal. In fact, if $\Omega_j$ is located at level $m-1$, then $\Omega_j\subset \Delta_m'\backslash \Delta_{m-1}'$ and it could be the case that $d_{\mathcal{C}}(\chi(\Omega_j),u_m)=d_{\mathcal{C}}(u_j,u_m)=1$. See Lemma 2.3 of \cite{CZ6}. 

\section{Proof of Theorem 1.2}
\setcounter{equation}{0}

We only treat the case where $m>0$. 
Theorem 1.2 was proved when $m=3, 4$ (by Theorem 1.1 of \cite{CZ3} and Theorem 1.1 of \cite{CZ5}). So we assume that $m\geq 5$. Note that all $u_j$, $j=1,2,\cdots,s$, are non-preperipheral geodesics, which allow us to acquire the configurations $(\tau_j, \Omega_j, \mathscr{U}_j)$ for $j=1,2,\cdots,s$. 

We first verify that $\Omega_1$ is located above or at level 1. Suppose not. Then $\Omega_1\cap \Delta_1'=\emptyset$ and there is no maximal element of $\mathscr{U}_1$ that equals $\Delta_1'$. There is a maximal element $\Delta_1\in \mathscr{U}_1$ such that $\Delta_1'\subset \Delta_1$. In particular, $\Delta_1\cap \Delta_0\neq \emptyset$, $\partial \Delta_1 \cap \partial \Delta_0=\emptyset$ and $\Delta_1\cup \Delta_0=\mathbf{H}$. This implies that $\mathbf{H}\backslash \Delta_1$ is disjoint from $\mathbf{H}\backslash \Delta_0$. So $\Omega_1$ is disjoint from $\Omega_0$. Hence by Lemma \ref{L2.1}, $d_{\mathcal{C}}(u_0,u_1)\geq 2$. This is a contradiction. 

By induction hypothesis, suppose that $\Omega_j$, $j\leq m-3$, is located above or at level $j$. We need to show that $\Omega_{j+1}$ is located above or at level $j+1$. Otherwise, suppose that $\Omega_{j+1}$ is located neither above nor at level $j+1$. There is a maximal element $\Delta_{j+1}''\in \mathscr{U}_{j+1}$ that contains $\Delta_{j+1}'(=g^{j+1}(\Delta_0^*)$), which says that $\partial \Delta_{j+1}''$ lies below $\overline{P_{j+1}Q_{j+1}}$. By assumption, $\Omega_j$ is located above or at level $j$. 

Case 1. $\Omega_j$ is located above level $j$ (Fig. 3, 4, 5, 6). By Lemma \ref{L2}, there is a maximal element $\Delta_j\in \mathscr{U}_j$, which covers the attracting fixed point $A$ of $g$, such that either $\partial \Delta_j$ lies above $\overline{P_{j+1}Q_{j+1}}$ or  $\partial \Delta_j$ intersects  $\overline{P_{j+1}Q_{j+1}}$. Both cases would imply that $\Delta_j\cap \Delta_{j+1}''\neq \emptyset$ and thus that $u_j$ and $u_{j+1}$ intersect. This contradicts that $d_{\mathcal{C}}(u_j,u_{j+1})=1$. 

Case 2. $\Omega_j$ is located at level $j$ (Fig. 2), then there is a maximal $\Delta_j\in \mathscr{U}_j$ such that $\Delta_j=\Delta_j'(=g^j(\Delta_0^*)$). Let $\Delta_j^*\in \mathscr{U}_j$ be the maximal element that contains $g(\mathbf{H}\backslash \Delta_j)$. Then either $\partial \Delta_j^*$ lies above $\overline{P_{j+1}Q_{j+1}}$, or $\partial \Delta_j^*=\overline{P_{j+1}Q_{j+1}}$. Note that $\partial \Delta_{j+1}''$ lies below $\overline{P_{j+1}Q_{j+1}}$. We conclude that in both cases $\Delta_j^*\cap \Delta_{j+1}''\neq \emptyset$. This again implies that $u_j$ and $u_{j+1}$ intersect, contradicting that $d_{\mathcal{C}}(u_j,u_ {j+1})=1$.

We conclude that  for all $j$ with $j\leq m-2$, $\Omega_j$ is located above or at level $j$. In particular, $\Omega_{m-2}$ is located above or at level $m-2$. If $\Omega_{m-2}$ is located above level $m-2$, then it lies above level $m-1$. By Lemma \ref{L3}, $d_{\mathcal{C}}(u_{m-2}, u_m)\geq 2$. If  $\Omega_{m-2}$ is located at level $m-2$, then again Lemma \ref{L3} says that $d_{\mathcal{C}}(u_{m-2}, u_m)\geq 2$. This proves that $s\geq m-1$ and thus that $d_{\mathcal{C}}(u_{0}, u_m)\geq m$.


\noindent {\em Remark. } From the proof we also deduce that $d_{\mathcal{C}}(u_{0}, u_m)= m$ if and only if $\Omega_0\cap \mbox{axis}(g)\neq \emptyset$ and $i(\tilde{c}, \tilde{u}_0)=1$. In this case, all $u_j$ are non preperipheral geodesic and for every $j=1,\cdots, m-1$, $\Omega_j$ is located at level $j$. Since $i(\tilde{c}, \tilde{u}_0)=1$, we see that $P_0=P_1$ and $Q_0=Q_1$. Also in the terminology of \cite{CZ6}, for $j=0,\cdots, m-1$, $\Omega_j$ is adjacent to $\Omega_{j+1}$, and thus $D(\Omega_j,\Omega_{j+1})=1$. It follows that  $d_{\mathcal{C}}(u_{0}, u_m)= \sum_{j=0}^{m-1}D(\Omega_j,\Omega_{j+1})=m$.


\begin{thebibliography}{99}

\bibitem[1]{Bers1} Bers, L., {\em Fiber spaces over Teichm\"{u}ller spaces. } Acta Math. 130 (1973), 89--126.

\bibitem[2]{Bir} Birman, J.S., {\em Braids, Links and Mapping class groups. } Ann of Math. Studies, No. 82, Princeton University Press, (1974).
\bibitem[3]{Bo} Bowditch, B., {\em Tight geodesics in the curve complex, } Invent. Math. 171 (2008), 281--300.


\bibitem[4]{F-L-M} Farb, B., Leininger, C., \& Margalit D., {\em The lower central series and pseudo-Anosov dilatations. } Amer. J. Math. 130 (2008), 799--827. 
\bibitem[5]{G-T} Gadre, V. \& Tsai, C., {\em Minimal pseudo-Anosov translation length on the complex of curves.} Geometry \& Topology, (2011), 1001--1017.
\bibitem[6]{H} Harvey, W. J., {\em Boundary structure of the modular group.} In {\em Riemann surfaces and related topics: Proceedings of the 1978 Stony Brook Conference}, Vol. 97 of Ann. of Math. Stud., 245--251, Princeton, N.J., 1981 Princeton Univ. Press. 


\bibitem[7]{Kr} Kra, I., {\em On the Nielsen-Thurston-Bers type of some self-maps of Riemann surfaces.} Acta Math. 146 (1981), 231--270.
 
\bibitem[8]{M-M} Masur, H., \& Minsky, Y., {\em Geometry of the complex of curves I: Hyperbolicity.} Invent.Math 138 (1999), 103-149.


\bibitem[9]{CZ0} Zhang, C.,  {\em Singularities of quadratic differentials and extremal Teichm\"{u}ller mappings defined by Dehn twists.} J. Aust. Math. Soc. 3 (2009), 275--288.


\bibitem[10]{CZ1} $\underline{\ \ \ \ \ \ \ \ \ }$, {\em Pseudo-Anosov maps and fixed points of boundary homeomorphisms compatible with a Fuchsian group.} Osaka J. Math, 46 (2009), 783--798. 



\bibitem[11]{CZ2} $\underline{\ \ \ \ \ \ \ \ \ }$, {\em On pseudo-Anosov maps with small dilatations on punctured Riemann spheres.} JP Journal of Geometry and Topology, 11 (2011), 117--145.
\bibitem[12]{CZ3} $\underline{\ \ \ \ \ \ \ \ \ }$, {\em Pseudo-Anosov maps and pairs of filling simple closed geodesics on Riemann surfaces.} Tokyo J. Math. 35 (2012).

\bibitem[13]{CZ3.5} $\underline{\ \ \ \ \ \ \ \ \ }$, {\em Pseudo-Anosov maps and pairs of filling simple closed geodesics on Riemann surfaces, II.} Tokyo J. Math. 36 (2013).

\bibitem[14]{CZ4} $\underline{\ \ \ \ \ \ \ \ \ }$, {\em Invariant Teichm\"{u}ller disks under hyperbolic mapping classes.} Hiroshima Math. J. 42 (2012), 169--187.

\bibitem[15]{CZ5} $\underline{\ \ \ \ \ \ \ \ \ }$, {\em On distances between curves in the curve complex and point-pushing pseudo-Anosov homeomorphisms. } JP Journal of Geometry and Topology, 2 (2012), 173--206.

\bibitem[16]{CZ6} $\underline{\ \ \ \ \ \ \ \ \ }$, {\em Tessellations of a hyperbolic plane by regions determined by vertices of the curve complex. } Preprint, 2013.

\end{thebibliography}
\end{document}